\documentclass[11pt]{article}%
\usepackage{amsfonts}
\usepackage{amsmath}
\usepackage{hyperref}
\usepackage{amssymb}
\usepackage{graphicx}%
\providecommand{\U}[1]{\protect\rule{.1in}{.1in}}
\newtheorem{theorem}{Theorem}

\newtheorem{remark}[theorem]{Remark}

\newenvironment{proof}[1][Proof]{\noindent\textbf{#1.} }{\ \rule{0.5em}{0.5em}}
\begin{document}

\title{{\large \textbf{A procedure for the change point problem in parametric models
based on }}$\phi$-{\large \textbf{divergence test-statistics}}}
\author{Batsidis, A.$^{1}$, Mart\'{\i}n, N.$^{2}$\thanks{Corresponding author, E-mail:
\href{mailto: nirian.martin@uc3m.es}{nirian.martin@uc3m.es}.}, Pardo, L.$^{3}$
and Zografos, K.$^{1}$\\$^{1}${\small Dept. Mathematics, University of Ioannina, Greece}\\$^{2}${\small Dept. Statistics, Carlos III University of Madrid, Spain}\\$^{3}${\small Dept. Statistics and O.R., Complutense University of Madrid,
Spain}}
\maketitle

\begin{abstract}
This paper studies the change point problem for a general parametric,
univariate or multivariate family of distributions. An information theoretic
procedure is developed which is based on general divergence measures for
testing the hypothesis of the existence of a change. For comparing the
accuracy of the new test-statistic a simulation study is performed for the
special case of a univariate discrete model. Finally, the procedure proposed
in this paper is illustrated through a classical change-point example.

\end{abstract}


\medskip


\section{Introduction\label{sec0}}

The change point problem has been considered and studied by several authors
the last five decades. Change point analysis is a statistical tool for
determining whether a change has taken place at a point of a sequence of
observations, such that the observations are described by one distribution up
to that point and by another distribution after that point. Change-point
analysis concerns with the detection and estimation of the point at which the
distribution changes. One change point problem or multiple change points
problem have been studied in the literature, depending on whether one or more
change points are observed in a sequence of random variables. Several methods,
parametric or non-parametric, have been developed to approach the solution of
this problem while the range of applications of change point analysis is
broad. Applications can be encountered in many areas such as statistical
quality control, public health, medicine, finance, biomedical signal
processing, meteorology, seismology, etc. The monograph by Chen and Gupta
(2000)\ summarizes recent developments in parametric change-point analysis.

Typical situations encountered in the literature of parametric multiple change
points analysis are as follows: Let $\boldsymbol{X}_{1},\boldsymbol{X}%
_{2},...,\boldsymbol{X}_{K}$ be $K$ independent $d$-variate observations
($d\in%
\mathbb{N}
$) and let $(\mathcal{X}^{(d)},\beta_{\mathcal{X}},P_{\boldsymbol{\theta}%
})_{\boldsymbol{\theta}\in\Theta}$ the statistical space associated with the
random variable (r.v.) $\boldsymbol{X}_{i}$, $i=1,...,K$. The probability
density function with respect to a $\sigma$-finite measure $\mu$ given by
$f_{\boldsymbol{\theta}_{i}}(\boldsymbol{x})=f(\boldsymbol{x}%
,\boldsymbol{\theta}_{i})=\frac{dP_{\boldsymbol{\theta}_{i}}}{d\mu}$,
$\boldsymbol{\theta}_{i}\in%
\mathbb{R}
^{m}$, $i=1,...,K$, $\boldsymbol{x\in%
\mathbb{R}
}^{d}$. For simplicity, $\mu$ is either the Lebesgue measure or a counting
measure. We adopt in the sequel the formulation of the multiple change point
problem as it appeared in Srivastava and Worsley (1986) and Chen and Gupta
(2000, 2004). Based on these authors, suppose that adjacent observations are
grouped in $q$ groups, so that $\boldsymbol{X}_{1},\boldsymbol{X}%
_{2},...,\boldsymbol{X}_{k_{1}}$, are in the first group, $\boldsymbol{X}%
_{k_{1}+1},\boldsymbol{X}_{k_{1}+2},...,\boldsymbol{X}_{k_{2}}$, are in the
second group and we continue in a similar manner until $\boldsymbol{X}%
_{k_{q-1}+1},\boldsymbol{X}_{k_{q-1}+2},...,\boldsymbol{X}_{k_{q}%
}=\boldsymbol{X}_{K}$ are in the $q$-th group.

Consider the model for changes in the parameters. This is formulated as a
problem of testing the following hypotheses,%
\begin{equation}
H_{0}:\text{ }\boldsymbol{\theta}_{1}=\boldsymbol{\theta}_{2}%
=...=\boldsymbol{\theta}_{K}\text{ (}=\boldsymbol{\theta}_{0}\text{,
}\boldsymbol{\theta}_{0}\text{ unknown),} \label{e1}%
\end{equation}
versus the alternative%
\[
H_{1}:\text{ }\boldsymbol{\theta}_{1}=...=\boldsymbol{\theta}_{k_{1}}%
\neq\boldsymbol{\theta}_{k_{1}+1}=...=\boldsymbol{\theta}_{k_{2}}\neq
...\neq\boldsymbol{\theta}_{k_{q-1}+1}=...=\boldsymbol{\theta}_{k_{q}%
}=\boldsymbol{\theta}_{K},
\]
where $q$, $1\leq q\leq K$, is the unknown number of changes and $k_{1}%
,k_{2},...,k_{q}$ are the unknown positions of the change points. The above
hypotheses can be equivalently stated in the form%
\begin{equation}
H_{0}:\boldsymbol{X}_{i}\text{ are described by }f_{\boldsymbol{\theta}_{0}%
}\text{, }i=1,...,K\text{ and }\boldsymbol{\theta}_{0}\text{ unknown,}
\label{e2}%
\end{equation}
versus the alternative%
\[
H_{1}:\boldsymbol{X}_{k_{j}+1},\boldsymbol{X}_{k_{j}+2},...,\boldsymbol{X}%
_{k_{j+1}},\text{ }j=0,...,q-1\text{ are described by }f_{\boldsymbol{\theta
}_{j+1}}\text{,}%
\]
with $\boldsymbol{X}_{k_{q}}=\boldsymbol{X}_{K}$.

There is an extensive bibliography on the subject and several methods to
search for the change point problem have appeared in the literature. Among
them, the generalized likelihood ratio test, Bayesian solution of the problem,
information criterion approaches, cumulative sum method, etc. Based on these
methods, several papers discuss the change-point problems in specific
probabilistic models, like the univariate and multivariate normal
distribution, the gamma model and the exponential model. For instance, Sen and
Srivastava (1980) focused on the single change-point problem. Moreover, they
consider that within each section, the distributions are the same, while the
distribution in a section is different from that in the preceding and the
following section in mean vector or covariance matrix. For an exposition of
these methods and their application to specific distributions we refer to the
monograph or the survey paper by Chen and Gupta (2000, 2001) and the
references appeared therein.

It has been proposed in these and other treatments (cf., for instance,
Vostrikova (1981)), that in order to study the multiple change point problem,
which is formulated by (\ref{e1}) or (\ref{e2}), we just need to test the
single change point hypothesis and then to repeat the procedure for each
subsequence. Hence, we turn to the testing of (\ref{e2}) against the
alternative,%
\begin{equation}
H_{1}:\text{ }\boldsymbol{X}_{i}\equiv f_{\boldsymbol{\theta}_{0}}\text{,
}i=1,...,\kappa\text{ \ and \ }\boldsymbol{X}_{i}\equiv f_{\boldsymbol{\theta
}_{1}}\text{, }i=\kappa+1,...,K, \label{e3}%
\end{equation}
where the symbol $\equiv$\ is used to denote that the observations on the left
follow the parametric density on the right. In (\ref{e3}), $\kappa$ represents
the position a single change point, which is supposed to be unknown. A general
description of this technique in the detection of the changes is summarized in
the following steps by Chen and Gupta (2001). First we test for no change
point versus one change point, that is, we test the null hypothesis given by
(\ref{e2}) versus the alternative given by (\ref{e3}) and equivalently stated
by $H_{1}$: $\boldsymbol{\theta}_{1}=...=\boldsymbol{\theta}_{\kappa}%
\neq\boldsymbol{\theta}_{\kappa+1}=...=\boldsymbol{\theta}_{K}$. Here,
$\kappa$ is the unknown location of the single change point. If $H_{0}$ is not
rejected, then the procedure is finished and there is no change point. If
$H_{0}$ is rejected, then there is a change point and we continue with the
step 2. In the second step we test separately the two subsequences before and
after the change point found in the first step for a change. In the sequel, we
repeat these two steps until no further subsequences have change points. At
the end of the procedure, the collection of change point locations found by
the previous steps constitute the set of the change points.

The subject of change point analysis is twofold. On the one hand to detect if
there is one or more changes in a sequence of observation. The second aspect
of change point analysis is the estimation of the number of changes and their
corresponding locations. In this paper we will develop an information
theoretic procedure which is based on divergence, in order to study the change
point problem. The measures background is a general parametric, univariate or
multivariate family of distributions. We describe formally the framework and
the problem in Section \ref{sec1}, and the main results are presented in
Section \ref{sec2}. In Section \ref{sec3} we focus our interest on a specific
distribution, the binomial distribution and a simulation study is performed in
order to compare the accuracy the new test-statistic with some pre-existing
test-statistics. In the final Section \ref{sec4}, the general results of this
paper are illustrated by means of the well-known Lisdisfarne scribes data set.

\section{Information theoretic procedure\label{sec1}}

Consider now the single change point problem, that is the problem of testing
the pair of hypotheses
\begin{subequations}
\begin{align}
H_{0}  &  :\text{ }\boldsymbol{X}_{i}\equiv f_{\boldsymbol{\theta}_{0}}\text{,
}i=1,...,K\label{e4}\\
H_{1}  &  :\text{ }\boldsymbol{X}_{i}\equiv f_{\boldsymbol{\theta}_{0}}\text{,
}i=1,...,\kappa\text{ and\ }\boldsymbol{X}_{i}\equiv f_{\boldsymbol{\theta
}_{1}}\text{, }i=\kappa+1,...,K, \label{eq4b}%
\end{align}
which are presented by (\ref{e2}) and (\ref{e3}), respectively. In the above
formulation, $\boldsymbol{\theta}_{0}$ and $\boldsymbol{\theta}_{1}$ are
unknown. Since $\kappa$ is the unknown location of the single change point, we
will consider all the candidate points $k\in\{1,...,K-1\}$. Let
$\widehat{\boldsymbol{\theta}}_{0,k}^{(K)}$ denotes the maximum likelihood
estimator (MLE) of $\boldsymbol{\theta}_{0}$ which is based on the random
sample $\boldsymbol{X}_{1},...,\boldsymbol{X}_{k}$ from $f_{\boldsymbol{\theta
}_{0}}$ and let $\widehat{\boldsymbol{\theta}}_{1,k}^{(K)}$ denotes the m.l.e.
of $\boldsymbol{\theta}_{1}$ which is based on the random sample
$\boldsymbol{X}_{k+1},...,\boldsymbol{X}_{K}$ from $f_{\boldsymbol{\theta}%
_{1}}$. If the hypothesis $H_{1}$ is true, then there is a difference between
the probabilistic models $f_{\widehat{\boldsymbol{\theta}}_{0,k}^{(K)}}$ and
$f_{\widehat{\boldsymbol{\theta}}_{1,k}^{(K)}}$, which cause a large value for
a measure of the distance between $f_{\widehat{\boldsymbol{\theta}}%
_{0,k}^{(K)}}$ and $f_{\widehat{\boldsymbol{\theta}}_{1,k}^{(K)}}$. Given that
the $\phi$-divergence is a broad family of distance measures between
probability distributions, the $\phi$-divergence between
$f_{\widehat{\boldsymbol{\theta}}_{0,k}^{(K)}}$ and
$f_{\widehat{\boldsymbol{\theta}}_{1,k}^{(K)}}$ is large if $H_{1}$ is true
and hence it can be used in order to decide if the candidate point $k$ in
(\ref{eq4b}) is a change point ($\kappa=k$). Taking into account that the
m.l.e. $\widehat{\boldsymbol{\theta}}_{0,k}^{(K)}$ and
$\widehat{\boldsymbol{\theta}}_{1,k}^{(K)}$ of $\theta_{0}$ and $\theta_{1}$,
respectively, depend on the candidate change point $k$, we will adopt the
following notation for the $\phi$-divergence between
$f_{\widehat{\boldsymbol{\theta}}_{0,k}^{(K)}}$ and
$f_{\widehat{\boldsymbol{\theta}}_{1,k}^{(K)}}$,
\end{subequations}
\begin{equation}
D_{\phi}^{(k)}=D_{\phi}^{(k)}\left(  f_{\widehat{\boldsymbol{\theta}}%
_{0,k}^{(K)}},f_{\widehat{\boldsymbol{\theta}}_{1,k}^{(K)}}\right)
=\int_{\mathcal{X}^{(d)}}f_{\widehat{\boldsymbol{\theta}}_{1,k}^{(K)}%
}(\boldsymbol{x})\phi\left(  \frac{f_{\widehat{\boldsymbol{\theta}}%
_{0,k}^{(K)}}(\boldsymbol{x})}{f_{\widehat{\boldsymbol{\theta}}_{1,k}^{(K)}%
}(\boldsymbol{x})}\right)  d\mu(\boldsymbol{x}), \label{e5}%
\end{equation}
provided that the convex function $\phi$ satisfies some additional conditions
(see page 408 in Pardo (2006)) which ensure the existence of the above
integral. Moreover, we consider convex functions $\phi$ which satisfy
$\phi(1)=0$ and $\phi^{\prime\prime}(1)\neq0$. Large values of $D_{\phi}%
^{(k)}$ support the existence of a change point and therefore large values of
$D_{\phi}^{(k)}$ suggest rejection of the null hypothesis $H_{0}$. Hence
$D_{\phi}^{(k)}$ can be used as a test statistic for testing the hypotheses
(\ref{e4}). Then, motivated by the fact that large values of $D_{\phi}^{(k)} $
are in favor of $H_{1}$, a test for testing the existence of a single change
point, that is the hypotheses (\ref{e4}), should be based on the $\phi
$-divergence test statistic,
\begin{equation}
T_{\phi}^{(K)}=\max_{k\in\{1,...,K-1\}}T_{\phi}^{(K)}(k), \label{e6}%
\end{equation}
where%
\begin{equation}
T_{\phi}^{(K)}(k)=\frac{k(K-k)}{K}\frac{2}{\phi^{\prime\prime}(1)}D_{\phi
}^{(k)}\left(  f_{\widehat{\boldsymbol{\theta}}_{0,k}^{(K)}}%
,f_{\widehat{\boldsymbol{\theta}}_{1,k}^{(K)}}\right)  . \label{e7}%
\end{equation}
Moreover, the unknown position of the change point $\kappa$ is estimated by
$\widehat{\kappa}_{\phi}$ such that%
\begin{equation}
\widehat{\kappa}_{\phi}=\underset{k\in\{1,...,K-1\}}{\arg\max}T_{\phi}%
^{(K)}(k)=\underset{k\in\{1,...,K-1\}}{\arg\max}\frac{k(K-k)}{K}D_{\phi}%
^{(k)}\left(  f_{\widehat{\boldsymbol{\theta}}_{0,k}^{(K)}}%
,f_{\widehat{\boldsymbol{\theta}}_{1,k}^{(K)}}\right)  . \label{e8}%
\end{equation}

Based on the above discussion, $H_{0}\ $in (\ref{e4}) is rejected for
$T_{\phi}^{(K)}>c$, where $c$ is a constant to be determined by the null
distribution of $T_{\phi}^{(K)}$. Hence, in order to use $T_{\phi}^{(K)}$ of
(\ref{e6}) for testing hypotheses (\ref{e4}), it is necessary the knowledge of
the distribution of $T_{\phi}^{(K)}$, under $H_{0}$.

There are two important reasons why working directly with test-statistics
$T_{\phi}^{(K)}$, defined in (\ref{e6}), is avoided, on one hand, its
asymptotic distribution $\sup_{t\in(0,1)}\frac{1}{t(1-t)}\left\Vert
\boldsymbol{W}_{0}^{(m)}(t)\right\Vert ^{2}$, is not an easy to handle random
variable (see for instance Theorem 1.2 and 1.3 in Gombay and Horv\'{a}th
(1989)) and on the other hand, in practice cases such that $\kappa
\in\{1,K-1\}$ are very difficult to detect. Let $N(\epsilon)$ be the set all
possible integers $k\in\{1,...,K-1\}$ such that $\frac{k}{K}\in\lbrack
\epsilon,1-\epsilon]$, with $\epsilon>0$, small enough. We shall modify
(\ref{e6}) to be maximized in $N(\epsilon)$, i.e.%
\begin{equation}
^{\epsilon}T_{\phi}^{(K)}=\max_{k\in N(\epsilon)}T_{\phi}^{(K)}(k),
\label{e10}%
\end{equation}
and in the same manner (\ref{e8}) becomes%
\begin{equation}
^{\epsilon}\widehat{\kappa}_{\phi}=\underset{k\in N(\epsilon)}{\arg\max
}T_{\phi}^{(K)}(k)=\underset{k\in N(\epsilon)}{\arg\max}\frac{k(K-k)}%
{K}D_{\phi}^{(k)}\left(  f_{\widehat{\boldsymbol{\theta}}_{0,k}^{(K)}%
},f_{\widehat{\boldsymbol{\theta}}_{1,k}^{(K)}}\right)  . \label{e11}%
\end{equation}

\section{Main result\label{sec2}}

In order to get the asymptotic distribution of the family of tests statistics
$T_{\phi}^{(K)}$, given in (\ref{e6}), we shall assume the usual regularity
assumptions for the multiparameter Central Limit Theorem (see for instance
Theorem 5.2.2. in Sen and Singer (1993)):

\begin{description}
\item[(i)] The parameter space, $\Theta$, is either $%
\mathbb{R}
^{m}$ or a rectangle in $%
\mathbb{R}
^{m}$.

\item[(ii)] For all $\boldsymbol{\theta}\neq\boldsymbol{\theta}^{\prime}%
\in\Theta\subset%
\mathbb{R}
^{m}$,%
\[
\mu\left(  \{\boldsymbol{x}\in\mathcal{X}^{(d)}:f_{\boldsymbol{\theta}%
}(\boldsymbol{x})\neq f_{\boldsymbol{\theta}^{\prime}}(\boldsymbol{x}%
)\}\right)  >0.
\]

\item[(iii)] For $\boldsymbol{\theta}=(\theta_{1},...,\theta_{m})^{T}$,%
\[
\frac{\partial}{\partial\theta_{i}}f_{\boldsymbol{\theta}}(\boldsymbol{x}%
)\quad\text{and}\quad\frac{\partial^{2}}{\partial\theta_{i}\partial\theta_{j}%
}f_{\boldsymbol{\theta}}(\boldsymbol{x})\text{, }i,j\in\{1,...,m\},
\]
exist almost everywhere and are such that%
\[
\left\vert \frac{\partial}{\partial\theta_{i}}f_{\boldsymbol{\theta}%
}(\boldsymbol{x})\right\vert \leq H_{i}(\boldsymbol{x})\quad\text{and}%
\quad\left\vert \frac{\partial^{2}}{\partial\theta_{i}\partial\theta_{j}%
}f_{\boldsymbol{\theta}}(\boldsymbol{x})\right\vert \leq G_{ij}(\boldsymbol{x}%
)\text{, }i,j\in\{1,...,m\},
\]
where%
\[
\int_{\mathcal{X}^{(d)}}H_{i}(\boldsymbol{x})d\mu(\boldsymbol{x})<\infty
\quad\text{and}\quad\int_{\mathcal{X}^{(d)}}G_{ij}(\boldsymbol{x}%
)d\mu(\boldsymbol{x})<\infty\text{, }i,j\in\{1,...,m\}.
\]

\item[(iv)] Denoting $\ell(\boldsymbol{x};\boldsymbol{\theta})=\log
f_{\boldsymbol{\theta}}(\boldsymbol{x})$,%
\[
\frac{\partial}{\partial\theta_{i}}\ell(\boldsymbol{x};\boldsymbol{\theta
})\quad\text{and}\quad\frac{\partial^{2}}{\partial\theta_{i}\partial\theta
_{j}}\ell(\boldsymbol{x};\boldsymbol{\theta})\text{, }i,j\in\{1,...,m\},
\]
exist almost everywhere and are such that the Fisher information matrix is
finite and positive definite. In addition, $\lim_{\delta\rightarrow0}%
\psi(\delta)=0$ where%
\[
\psi(\delta)=E_{\boldsymbol{\theta}}\left[  \sup_{\{\boldsymbol{h}:\left\Vert
\boldsymbol{h}\right\Vert \leq\delta\}}\left\Vert \frac{\partial^{2}}%
{\partial\boldsymbol{\theta}\partial\boldsymbol{\theta}^{T}}\ell
(\boldsymbol{x};\boldsymbol{\theta+h})-\frac{\partial^{2}}{\partial
\boldsymbol{\theta}\partial\boldsymbol{\theta}^{T}}\ell(\boldsymbol{x}%
;\boldsymbol{\theta})\right\Vert \right]  ,
\]
with $\frac{\partial^{2}}{\partial\boldsymbol{\theta}\partial
\boldsymbol{\theta}^{T}}\ell(\boldsymbol{x};\boldsymbol{\theta})=\left(
\frac{\partial^{2}}{\partial\theta_{i}\partial\theta_{j}}\ell(\boldsymbol{x}%
;\boldsymbol{\theta})\right)  _{i,j\in\{1,...,m\}}$ and $\left\Vert
\boldsymbol{\bullet}\right\Vert $ is the Euclidean norm.
\end{description}

\begin{theorem}
Under $H_{0}$ in (\ref{e4}) and the previous regularity assumptions, (i)-(iv),
the asymptotic distribution of (\ref{e10}) is given by%
\begin{equation}
^{\epsilon}T_{\phi}^{(K)}\overset{\mathcal{L}}{\underset{K\rightarrow
\infty}{\longrightarrow}}\mathcal{T}_{m,\epsilon} \label{e9}%
\end{equation}
where $m=\dim(\Theta)$,%
\begin{equation}
\mathcal{T}_{m,\epsilon}=\sup_{t\in\lbrack\epsilon,1-\epsilon]}\frac
{1}{t(1-t)}\left\Vert \boldsymbol{W}_{0}^{(m)}(t)\right\Vert ^{2},
\label{e10b}%
\end{equation}
with $\boldsymbol{W}_{0}^{(m)}(t)=\{(W_{0,1}(t),...,W_{0,m}(t))^{T}%
\}_{t\in\lbrack0,1]}$,\ being an $m$-dimensional vector of independent
Brownian bridges\ and $\left\Vert \boldsymbol{W}_{0}^{(m)}(t)\right\Vert
^{2}=\sum_{i=1}^{m}W_{0,i}^{2}(t)$.
\end{theorem}

\begin{proof}
According to the properties of the MLEs we know that%
\begin{align*}
&  \sqrt{k}\left(  \widehat{\boldsymbol{\theta}}_{0,k}^{(K)}%
-\boldsymbol{\theta}_{0}\right)  \overset{\mathcal{L}}{\underset{k\rightarrow
\infty}{\longrightarrow}}\mathcal{N}\left(  0,\boldsymbol{I}_{\mathcal{F}%
}(\boldsymbol{\theta}_{0})^{-1}\right)  ,\\
&  \sqrt{K-k}\left(  \widehat{\boldsymbol{\theta}}_{1,k}^{(K)}%
-\boldsymbol{\theta}_{1}\right)  \overset{\mathcal{L}%
}{\underset{(K-k)\rightarrow\infty}{\longrightarrow}}\mathcal{N}\left(
0,\boldsymbol{I}_{\mathcal{F}}(\boldsymbol{\theta}_{1})^{-1}\right)  ,
\end{align*}
where for $\boldsymbol{\theta\in}\Theta$, such that $m=\dim\Theta$,
$\boldsymbol{I}_{\mathcal{F}}(\boldsymbol{\theta})=\left(  -E\left[
\frac{\partial^{2}}{\partial\theta_{i}\partial\theta_{i}}\log
f_{\boldsymbol{\theta}}(\boldsymbol{X}_{1})\right]  \right)  _{i,j\in
\{1,...,m\}}$, is the information matrix. If we consider that $\lambda
_{k}^{(K)}=\lim_{K\rightarrow\infty}\frac{k}{K}$, then%
\begin{align*}
&  \sqrt{\frac{k(K-k)}{K}}\left(  \widehat{\boldsymbol{\theta}}_{0,k}%
^{(K)}-\boldsymbol{\theta}_{0}\right)  \overset{\mathcal{L}%
}{\underset{K\rightarrow\infty}{\longrightarrow}}\mathcal{N}\left(
0,(1-\lambda_{k}^{(K)})\boldsymbol{I}_{\mathcal{F}}(\boldsymbol{\theta}%
_{0})^{-1}\right)  ,\\
&  \sqrt{\frac{k(K-k)}{K}}\left(  \widehat{\boldsymbol{\theta}}_{1,k}%
^{(K)}-\boldsymbol{\theta}_{1}\right)  \overset{\mathcal{L}%
}{\underset{K\rightarrow\infty}{\longrightarrow}}\mathcal{N}\left(
0,\lambda_{k}^{(K)}\boldsymbol{I}_{\mathcal{F}}(\boldsymbol{\theta}_{1}%
)^{-1}\right)  .
\end{align*}
This means that under $\mathcal{H}_{0}$, i.e. $\boldsymbol{\theta}%
_{0}=\boldsymbol{\theta}_{1}$,
\[
\sqrt{\frac{k(K-k)}{K}}\left(  \widehat{\boldsymbol{\theta}}_{0,k}%
^{(K)}-\widehat{\boldsymbol{\theta}}_{1,k}^{(K)}\right)  \overset{\mathcal{L}%
}{\underset{K\rightarrow\infty}{\longrightarrow}}\mathcal{N}\left(
0,\boldsymbol{I}_{\mathcal{F}}(\boldsymbol{\theta}_{0})^{-1}\right)  ,
\]
and hence we can construct a Wald-type test-statistic as follows%
\begin{equation}
Q_{k}^{(K)}=\frac{k(K-k)}{K}\left(  \widehat{\boldsymbol{\theta}}_{0,k}%
^{(K)}-\widehat{\boldsymbol{\theta}}_{1,k}^{(K)}\right)  ^{T}%
\widehat{\boldsymbol{I}_{\mathcal{F}}(\boldsymbol{\theta}_{0})}\left(
\widehat{\boldsymbol{\theta}}_{0,k}^{(K)}-\widehat{\boldsymbol{\theta}}%
_{1,k}^{(K)}\right)  , \label{eq}%
\end{equation}
where $\widehat{\boldsymbol{I}_{\mathcal{F}}(\boldsymbol{\theta}_{0})}$ is any
consistent estimator of $\boldsymbol{I}_{\mathcal{F}}(\boldsymbol{\theta}%
_{0})$. From Theorem 1 in Hawkins (1987) we know that%
\[
\max_{k\in N(\epsilon)}Q_{k}^{(K)}\overset{\mathcal{L}}{\underset{K\rightarrow
\infty}{\longrightarrow}}\mathcal{T}_{m,\epsilon}%
\]
In addition from Pardo (2006), page 443, we have%
\[
T_{\phi}^{(K)}(k)=\frac{k(K-k)}{K}\frac{2}{\phi^{\prime\prime}(1)}D_{\phi
}(f_{\widehat{\boldsymbol{\theta}}_{0,k}^{(K)}},f_{\widehat{\boldsymbol{\theta
}}_{1,k}^{(K)}})=Q_{k}^{(K)}+o_{P}(1)
\]
where%
\[
D_{\phi}(f_{\widehat{\boldsymbol{\theta}}_{0,k}^{(K)}}%
,f_{\widehat{\boldsymbol{\theta}}_{1,k}^{(K)}})=\int
f_{\widehat{\boldsymbol{\theta}}_{1,k}^{(K)}}(x)\phi\left(  \frac
{f_{\widehat{\boldsymbol{\theta}}_{0,k}^{(K)}}(x)}%
{f_{\widehat{\boldsymbol{\theta}}_{1,k}^{(K)}}(x)}\right)  dx.
\]
With both results we conclude (\ref{e9}).
\end{proof}

\begin{remark}
If we compare (\ref{eq}) with formula (2.3) in Hawkins (1987), both apparently
are not equivalent because in our case $\frac{k(K-k)}{K}$ appears rather than
$k(K-k)$ of formula (2.3). This difference is associated with the way of
understanding Fisher Information matrix, in fact our Wald test-statistic
coincide with the empirical stochastic process denoted by $\widetilde{Q}%
_{K}(t)$\ at the beginning of Section 3 in Hawkins (1987).
\end{remark}

\begin{remark}
The probability distribution function of random variable $\mathcal{T}%
_{m,\epsilon}$, for $\epsilon>0$, given in (\ref{e10b}), can be found in Sen
(1981, page 397) and De Long (1981). The computation of the probability
distribution function is complex, however it is possible to approximate the
$p$-value of the test in which the distribution of $\mathcal{T}_{m,\epsilon}$
is considered under the null hypothesis. In Estrella (2003), for instance,%
\begin{equation}
\widetilde{p\mathrm{-value}}(x,\epsilon)=\frac{1}{\Gamma\left(  \frac{m}%
{2}\right)  }\left(  \frac{x}{2}\right)  ^{\frac{m}{2}}\exp\left\{  -\frac
{x}{2}\right\}  \left(  \log\left(  \frac{(1-\epsilon)^{2}}{\epsilon^{2}%
}\right)  \left(  1-\frac{m}{x}\right)  +\frac{2}{x}\right)  , \label{eqAprox}%
\end{equation}
with $\Gamma\left(  t\right)  $\ being the Gamma function, is proposed as an
approximation of%
\begin{align}
p\mathrm{-value}(x,\epsilon)  &  =\Pr\left(  \mathcal{T}_{m,\epsilon
}>x\right)  =\Pr\left(  \sup_{s\in\left(  1,(1-\epsilon)^{2}/\epsilon
^{2}\right)  }\frac{1}{\sqrt{s}}\left\Vert \boldsymbol{W}_{0}^{(m)}%
(s)\right\Vert >\sqrt{x}\right) \nonumber\\
&  =\frac{1}{\Gamma\left(  \frac{m}{2}\right)  }\left(  \frac{x}{2}\right)
^{\frac{m}{2}}\exp\left\{  -\frac{x}{2}\right\}  \left(  \log\left(
\frac{(1-\epsilon)^{2}}{\epsilon^{2}}\right)  \left(  1-\frac{m}{x}\right)
+\frac{2}{x}+O\left(  \frac{1}{x^{2}}\right)  \right)  .\nonumber
\end{align}
When calibrating the approximation for the univariate parameter ($m=1$), we
can take into account that the exact quantiles of order $(1-\alpha
)\in\{0.90,0.95,0.99\}$ for $\epsilon=0.05$, are $8.31$, $9.90$ and
$13.45$\ respectively, i.e. $p\mathrm{-value}(8.31,0.05)=0.1$,
$p\mathrm{-value}(9.90,0.05)=0.05$, $p\mathrm{-value}(13.45,0.05)=0.01$. If we
use (\ref{eqAprox}) with $\epsilon=0.05$ and the aforementioned quantiles, we
obtain $\widetilde{p\mathrm{-value}}(8.31,0.05)=9.\,\allowbreak778\,9\times
10^{-2}$, $\widetilde{p\mathrm{-value}}(9.90,0.05)=4.\,\allowbreak
886\,8\times10^{-2}$, $\widetilde{p\mathrm{-value}}%
(13.45,0.05)=9.\,\allowbreak835\,8\times10^{-3}$. We can see that in
particular, $\widetilde{p\mathrm{-value}}(x,0.05)$ approximates very well
$p\mathrm{-value}(x,0.05)$ when $x$ is the quantile of order $1-\alpha=0.99$,
which is in practice of major interest.
\end{remark}

\section{Simulation Study\label{sec3}}

In this section we are going to focus on the change point analysis for a
particular discrete probability model, the binomial model. For this special
case we will give an explicit expression for divergence based test-statistics.
The accuracy will be compared by simulation with respect to pre-existing
test-statistics. In this context, suppose we are dealing with a sequence of
independent r.v.'s $\widetilde{X}_{i}\sim\mathcal{B}\mathrm{in}(n_{i}%
,\theta_{i})$, $i=1,...,K$, for which we are interested in testing (\ref{e1}).
In order to do that we are going to consider a sequence of independent
Bernoulli r.v.'s $X_{ih}\sim\mathcal{B}\mathrm{er}(\theta_{i})$, $i=1,...,K$,
$h=1,...,n_{i}$, whose probability mass function (p.m.f.) is given by
$p_{\theta_{i}}(x)=\theta_{i}^{x}(1-\theta_{i})^{1-x}$,\ $x\in\{0,1\}$, and
$p_{\theta_{i}}(x)=0$,\ $x\notin\{0,1\}$. If we denote the cumulative steps
between consecutive Binomial r.v.'s by
\[
N_{k}=\sum_{i=1}^{k}n_{i}.
\]
the change points are located at $\{1,2,...,N_{K}-1,N_{K}\}$ for $X_{ih}$ and
at $\{N_{k}\}_{k=1}^{K}$\ for $\widetilde{X}_{i}$. Hence, $X_{ih}$ is the only
sequence of r.v.'s which are strictly identically distributed, but the change
points of interest are located in $\{N_{k}\}_{k=1}^{K}\subset\{1,2,...,N_{K}%
-1,N_{K}\}$. This means that we can construct the test-statistic by
considering a sequence of i.i.d. r.v.'s but in addition we restrict the set of
possible change points to $\{N_{k}\}_{k=1}^{K}$, rather than one step change
points. When the change point is located at $N_{k}$, the MLEs of $\theta_{0}$
and $\theta_{1}$\ are given by%
\begin{align*}
\widehat{\theta}_{0,k}^{(K)}  &  =\frac{Y_{k}}{N_{k}},\qquad\widehat{\theta
}_{1,k}^{(K)}=\frac{Y_{K}-Y_{k}}{N_{K}-N_{k}},\\
Y_{k}  &  =\sum_{i=1}^{k}\widetilde{X}_{i}=\sum_{i=1}^{k}\sum_{h=1}^{n_{i}%
}X_{ih}.
\end{align*}

The likelihood ratio test-statistic is given by $S^{(K)}=\max_{k\in
\{1,...,K\}}S_{k}^{(K)}$, where%
\begin{align}
S_{k}^{(K)}  &  =2\left[  N_{k}\left(  \widehat{\theta}_{0,k}^{(K)}\log\left(
\frac{\widehat{\theta}_{0,k}^{(K)}}{\widehat{\theta}_{0,K}^{(K)}}\right)
+(1-\widehat{\theta}_{0,k}^{(K)})\log\left(  \frac{1-\widehat{\theta}%
_{0,k}^{(K)}}{\widehat{\theta}_{0,K}^{(K)}}\right)  \right)  \right.
\nonumber\\
&  \left.  +(N_{K}-N_{k})\left(  \widehat{\theta}_{1,k}^{(K)}\log\left(
\frac{\widehat{\theta}_{1,k}^{(K)}}{\widehat{\theta}_{0,K}^{(K)}}\right)
+(1-\widehat{\theta}_{1,k}^{(K)})\log\left(  \frac{1-\widehat{\theta}%
_{1,k}^{(K)}}{\widehat{\theta}_{0,K}^{(K)}}\right)  \right)  \right]
\label{LRT}%
\end{align}
Two important papers which cover $S^{(K)}$ are Worsley (1983), and Horv\'{a}th
(1989). The expression they gave for $S_{k}^{(K)}$ is not exactly the same,
but it is equivalent to (\ref{LRT}) (see formula (3.22) in Horv\'{a}th and
Serbinowska (1995)). Horv\'{a}th (1989) found that the asymptotic distribution
for a kind of normalization of $S^{(K)}$ based on the Darling-Erd\"{o}s
formula%
\[
G^{(K)}=\sqrt{2\log N_{K}S^{(K)}}-2\log N_{K}-\frac{1}{2}\log\log N_{K}%
+\frac{1}{2}\log\pi,
\]
is asymptotically equal to a Extreme Value random variable with parameters
$\mu=\log2$ and $\beta=1$. In addition, in Theorem 1.2 of Horv\'{a}th and
Serbinowska (1995), a modified version of the likelihood ratio test-statistic
was given, $\widetilde{S}^{(K)}=\max_{k\in\{1,...,K\}}^{(K)}\widetilde{S}%
_{k}^{(K)}$, where%
\[
\widetilde{S}_{k}^{(K)}=\frac{N_{k}(N_{K}-N_{k})}{N_{K}^{2}}S_{k}^{(K)}.
\]
The asymptotic distribution of $\widetilde{S}^{(K)}$ is the supremum in
$(0,1)$ of a standard univariate Brownian bridge (its probability distribution
function is tabulated in Kiefer (1959)). We consider the version of the Wald
test-statistic $^{\epsilon}Q^{(K)}=\max_{k\in N(\epsilon)}Q_{k}^{(K)}$, with%
\[
Q_{k}^{(K)}=\frac{N_{k}(N_{K}-N_{k})}{N_{K}}(\widehat{\theta}_{0,k}%
^{(K)}-\widehat{\theta}_{1,k}^{(K)})^{2}\widehat{\boldsymbol{I}_{\mathcal{F}%
}(\boldsymbol{\theta}_{0})},
\]
where the consistent estimator of $\boldsymbol{I}_{\mathcal{F}}%
(\boldsymbol{\theta}_{0})$ is given by%
\[
\widehat{\boldsymbol{I}_{\mathcal{F}}(\boldsymbol{\theta}_{0})}=\frac{N_{k}%
}{N_{K}}\boldsymbol{I}_{\mathcal{F}}(\widehat{\theta}_{0,k}^{(K)})+\frac
{N_{K}-N_{k}}{N_{K}}\boldsymbol{I}_{\mathcal{F}}(\widehat{\theta}_{1,k}%
^{(K)})=\frac{N_{k}}{N_{K}}\frac{1}{\widehat{\theta}_{0,k}^{(K)}\left(
1-\widehat{\theta}_{0,k}^{(K)}\right)  }+\frac{N_{K}-N_{k}}{N_{K}}\frac
{1}{\widehat{\theta}_{1,k}^{(K)}\left(  1-\widehat{\theta}_{1,k}^{(K)}\right)
}.
\]
Finally, in order to give an explicit expression for divergence based
test-statistics we are going to focus on a family of divergences, power
divergences (see Read and Cressie (1988)), for which $\phi_{\lambda}\left(
x\right)  =\frac{1}{\lambda(1+\lambda)}\left(  x^{\lambda+1}-x-\lambda
(x-1)\right)  $, if $\lambda(1+\lambda)\neq0$ and $\phi_{\lambda}\left(
x\right)  =\lim_{\ell\rightarrow\lambda}\phi_{\ell}\left(  x\right)  $, if
$\lambda(1+\lambda)=0$, that is for each $\lambda\in%
\mathbb{R}
$ we obtain a different divergence measure between the p.m.f.s $p_{\theta_{0}%
}$ and $p_{\theta_{1}}$,%
\[
D_{\lambda}(p_{\theta_{0}},p_{\theta_{1}})=\frac{1}{\lambda(1+\lambda)}\left(
\frac{\theta_{0}^{\lambda+1}}{\theta_{1}^{\lambda}}+\frac{(1-\theta
_{0})^{\lambda+1}}{(1-\theta_{1})^{\lambda}}-1\right)  \text{, if }%
\lambda(1+\lambda)\neq0.
\]
When $\lambda=0$ the power divergence coincides with the so called Kullback
divergence%
\[
D_{0}(p_{\theta_{0}},p_{\theta_{1}})=D_{\mathrm{Kull}}(p_{\theta_{0}%
},p_{\theta_{1}})=\left(  \theta_{0}\log\left(  \frac{\theta_{0}}{\theta_{1}%
}\right)  +(1-\theta_{0})\log\left(  \frac{1-\theta_{0}}{1-\theta_{1}}\right)
\right)  ,
\]
and when $\lambda=-1$ the power divergence coincides with the modified
Kullback divergence $D_{-1}(p_{\theta_{0}},p_{\theta_{1}})=D_{\mathrm{Kull}%
}(p_{\theta_{1}},p_{\theta_{0}})$. Hence, the shape of the power-divergence
based test-statistics is $^{\epsilon}T_{\lambda}^{(K)}=$\linebreak$\max_{k\in
N(\epsilon)}T_{\lambda}(\widehat{\theta}_{0,k}^{(K)},\widehat{\theta}%
_{1,k}^{(K)})$, where%
\begin{equation}
T_{\lambda}(\widehat{\theta}_{0,k}^{(K)},\widehat{\theta}_{1,k}^{(K)}%
)=2\frac{N_{k}(N_{K}-N_{k})}{N_{K}}D_{\lambda}\left(  p_{\widehat{\theta
}_{0,k}^{(K)}},p_{\widehat{\theta}_{1,k}^{(K)}}\right)  ,\nonumber
\end{equation}
that is%
\begin{equation}
T_{\lambda}(\widehat{\theta}_{0,k}^{(K)},\widehat{\theta}_{1,k}^{(K)}%
)=\frac{N_{k}(N_{K}-N_{k})}{N_{K}}\frac{2}{\lambda(1+\lambda)}\left(
\frac{\left(  \widehat{\theta}_{0,k}^{(K)}\right)  ^{\lambda+1}}{\left(
\widehat{\theta}_{1,k}^{(K)}\right)  ^{\lambda}}+\frac{\left(
1-\widehat{\theta}_{0,k}^{(K)}\right)  ^{\lambda+1}}{\left(  1-\widehat{\theta
}_{1,k}^{(K)}\right)  ^{\lambda}}-1\right)  \text{, for }\lambda
(1+\lambda)\neq0, \label{phid}%
\end{equation}
and%
\begin{equation}
T_{0}(\widehat{\theta}_{0,k}^{(K)},\widehat{\theta}_{1,k}^{(K)})=2\frac
{N_{k}(N_{K}-N_{k})}{N_{K}}\left(  \widehat{\theta}_{0,k}^{(K)}\log\left(
\frac{\widehat{\theta}_{0,k}^{(K)}}{\widehat{\theta}_{1,k}^{(K)}}\right)
+(1-\widehat{\theta}_{0,k}^{(K)})\log\left(  \frac{1-\widehat{\theta}%
_{0,k}^{(K)}}{1-\widehat{\theta}_{1,k}^{(K)}}\right)  \right)  . \label{Kull}%
\end{equation}
Assuming that there is a monotone, continuous function $g$ such that $g(0)=0$
and%
\[
\lim_{K\rightarrow\infty}\max_{k\in N(\epsilon)}\left\vert \frac{N_{k}%
(N_{K}-N_{k})}{N_{K}}-g\left(  \frac{k(K-k)}{K}\right)  \right\vert =0,
\]
the asymptotic distribution of $^{\epsilon}Q^{(K)}$ and $^{\epsilon}%
T_{\lambda}^{(K)}$, for all $\lambda\in%
\mathbb{R}
$, is the supremum in $[\epsilon,1-\epsilon]$ of the univariate tied-down
Bessel process, i.e. (\ref{e10b}) with $m=1$. This assumption is very similar
to the assumption given in Horv\'{a}th and Serbinowska (1995) for the
asymptotic distribution of $\widetilde{S}^{(K)}$.%

\begin{table}[htbp] \centering
\caption{Exact simulated sizes.\label{Table1}}\smallskip\
\begin{tabular}
[c]{|llccccccc|}\hline
&  & \multicolumn{2}{c}{$K=64$} & \multicolumn{2}{c}{$K=300$} &
\multicolumn{2}{c}{$K=500$} & $K=\infty$\\\hline
& $1-\alpha$ & $x_{1-\alpha}$ & $\widehat{\alpha}$ & $x_{1-\alpha}$ &
$\widehat{\alpha}$ & $x_{1-\alpha}$ & $\widehat{\alpha}$ & $x_{1-\alpha}%
$\\\hline
& \multicolumn{1}{|l}{$0.90$} & \multicolumn{1}{|c}{$1.302$} & $0.0664$ &
$1.386$ & $0.0786$ & $1.420$ & $0.0860$ & $1.498$\\
$\widetilde{S}^{(K)}$ & \multicolumn{1}{|l}{$0.95$} &
\multicolumn{1}{|c}{$1.619$} & $0.0318$ & $1.710$ & $0.0372$ & $1.740$ &
$0.0400$ & $1.844$\\
& \multicolumn{1}{|l}{$0.99$} & \multicolumn{1}{|c}{$2.595$} & $0.0094$ &
$2.484$ & $0.0072$ & $2.531$ & $0.0074$ & $2.649$\\\hline
& \multicolumn{1}{|l}{$0.90$} & \multicolumn{1}{|c}{$1.707$} & $0.0208$ &
$1.939$ & $0.0260$ & $2.011$ & $0.0288$ & $2.943$\\
$G^{(K)}$ & \multicolumn{1}{|l}{$0.95$} & \multicolumn{1}{|c}{$2.277$} &
$0.0076$ & $2.431$ & $0.0094$ & $2.555$ & $0.0118$ & $3.663$\\
& \multicolumn{1}{|l}{$0.99$} & \multicolumn{1}{|c}{$3.394$} & $0.0002$ &
$3.653$ & $0.0002$ & $3.796$ & $0.0000$ & $5.293$\\\hline
& \multicolumn{1}{|l}{$0.90$} & \multicolumn{1}{|c}{$7.351$} & $0.0654$ &
$7.881$ & $0.0834$ & $7.801$ & $0.0824$ & $8.31$\\
$^{0.05}T_{0}^{(K)}$ & \multicolumn{1}{|l}{$0.95$} &
\multicolumn{1}{|c}{$8.968$} & $0.0340$ & $9.458$ & $0.0412$ & $9.514$ &
$0.0430$ & $9.90$\\
& \multicolumn{1}{|l}{$0.99$} & \multicolumn{1}{|c}{$12.730$} & $0.0086$ &
$13.143$ & $0.0094$ & $12.981$ & $0.0078$ & $13.45$\\\hline
& \multicolumn{1}{|l}{$0.90$} & \multicolumn{1}{|c}{$7.374$} & $0.0664$ &
$7.884$ & $0.0832$ & $7.809$ & $0.0828$ & $8.31$\\
$^{0.05}T_{1}^{(K)}$ & \multicolumn{1}{|l}{$0.95$} &
\multicolumn{1}{|c}{$9.007$} & $0.0352$ & $9.464$ & $0.0412$ & $9.509$ &
$0.0432$ & $9.90$\\
& \multicolumn{1}{|l}{$0.99$} & \multicolumn{1}{|c}{$12.851$} & $0.0088$ &
$13.128$ & $0.0094$ & $12.981$ & $0.0078$ & $13.45$\\\hline
& \multicolumn{1}{|l}{$0.90$} & \multicolumn{1}{|c}{$7.432$} & $0.0688$ &
$7.911$ & $0.0840$ & $7.818$ & $0.0834$ & $8.31$\\
$^{0.05}T_{2}^{(K)}$ & \multicolumn{1}{|l}{$0.95$} &
\multicolumn{1}{|c}{$9.141$} & $0.0370$ & $9.495$ & $0.0416$ & $9.519$ &
$0.0434$ & $9.90$\\
& \multicolumn{1}{|l}{$0.99$} & \multicolumn{1}{|c}{$13.061$} & $0.0094$ &
$13.129$ & $0.0094$ & $13.015$ & $0.0084$ & $13.45$\\\hline
& \multicolumn{1}{|l}{$0.90$} & \multicolumn{1}{|c}{$7.311$} & $0.0642$ &
$7.871$ & $0.0822$ & $7.800$ & $0.0824$ & $8.31$\\
$^{0.05}Q^{(K)}$ & \multicolumn{1}{|l}{$0.95$} & \multicolumn{1}{|c}{$8.934$}
& $0.0334$ & $9.441$ & $0.0408$ & $9.508$ & $0.0430$ & $9.90$\\
& \multicolumn{1}{|l}{$0.99$} & \multicolumn{1}{|c}{$12.742$} & $0.0084$ &
$13.135$ & $0.0094$ & $12.973$ & $0.0078$ & $13.45$\\\hline
\end{tabular}%
\end{table}%

A simulation study is performed in order to compare the accuracy of the
proposed power divergence type test with respect to pre-existing
test-statistics. In this context we apply test-statistics $\widetilde{S}%
^{(K)}$, $G^{(K)}$, $^{0.05}T_{0}^{(K)}$, $^{0.05}T_{1}^{(K)}$, $^{0.05}%
T_{2}^{(K)}$, $^{0.05}Q^{(K)}$ with $5000$ replication. The design is
essentially the same as the study performed in Horv\'{a}th and Serbinowska
(1995): $\theta_{0}=0.5$; three possible values of $K$ and nominal sizes
$\alpha$\ are considered; apart from the quantiles of order $1-\alpha$,
$x_{1-\alpha}$, the exact sizes $\widehat{\alpha}$ are calculated. With
$K=\infty$, it is understood that $x_{1-\alpha}$ is the asymptotic quantile
associated to the corresponding test-statistic. Taking into account that the
maximization for obtaining $^{\epsilon}T_{1}^{(K)}$, $^{\epsilon}T_{2}^{(K)}$,
$^{\epsilon}Q^{(K)}$, with $\epsilon=0.05$ is over all possible integers $k\in
N(\epsilon)$, we removed $k\in\{1,...,K-1\}$ when $k<\epsilon K$ or
$k>(1-\epsilon)K$.

Looking at the results given in Table \ref{Table1}, the worst approximation of
$\alpha$ is obtained with $G^{(K)}$. The Wald test-statistic $^{0.05}Q^{(K)}$
is a good competitor for the test-statistic introduced in Horv\'{a}th and
Serbinowska (1995), $\widetilde{S}^{(K)}$. All the exact sizes underestimate
the nominal size, which means that the best approximation is obtained with the
greatest value of $\widehat{\alpha}$, hardly ever obtained with the
power-divergence based test-statistic with $\lambda=2$, $^{0.05}T_{2}^{(K)}%
$.\newpage

\section{Numerical Example: Lindisfarne Scribes problem\label{sec4}}

The Lindisfarne Gospels are presumed to be the work of a monk named Eadfrith,
who became Bishop of Lindisfarne in year 698. In the 10th century an Old
English translation of the Gospels was made for one or more scribes. Several
statisticians have been devoted to studying the problem of the number of
scribes who participated in the translation of the Gospels. Such a problem is
known as the \textquotedblleft Lindisfarne Scribes problem\textquotedblright.

In the framework of the model that is followed in the simulation study, the
Lindisfarne Gospels are considered to be divided into $K=64$ consecutive
sections (see Ross (1950) for more details). It is supposed that each section
could have been translated by one scribe and the same scribe is associated
only with consecutive sections. Since the present indicative in Old English
verbs admitted several variants in its spelling, the custom of using these
variants can be used as a key factor useful to identifying different
translators. Based on the data given in Table \ref{Table2}, it is counted
$n_{i}$ as the total of observed frequencies that the third singular or second
plural appears in each section $i=1,...,64$, and the observation $x_{i}$
(coming from r.v. $X_{i}$) represents how many times ending $-s$ appear in
these verbs. Note that either the third singular or second plural admit two
endings, $-s$ and $-\delta$, and hence if we want to know how many times
ending $-\delta$ appear in these verbs, the observations are obtained as
$n_{i}-x_{i}$, $i=1,...,K$. It is assumed that the custom of using both
endings for each scribe is different and for this reason our interest is to
find the consecutive changes in the probability structure of both endings.%

\begin{table}[htbp] \centering
\caption{Data of the Lindisfarne's problem\label{Table2}}\smallskip\
\begin{tabular}
[c]{|lrrrrrrrrrrrrrrrr|}\hline
\multicolumn{1}{|l|}{$i$} & 1 & 2 & 3 & 4 & 5 & 6 & 7 & 8 & 9 & 10 & 11 & 12 &
13 & 14 & 15 & 16\\
\multicolumn{1}{|l|}{$x_{i}$} & 12 & 29 & 31 & 21 & 14 & 41 & 49 & 30 & 39 &
35 & 26 & 32 & 30 & 17 & 19 & 33\\
\multicolumn{1}{|l|}{$n_{i}$} & 21 & 39 & 44 & 25 & 19 & 66 & 62 & 34 & 47 &
47 & 29 & 33 & 38 & 21 & 21 & 36\\\hline
\multicolumn{1}{|l|}{$i$} & 17 & 18 & 19 & 20 & 21 & 22 & 23 & 24 & 25 & 26 &
27 & 28 & 29 & 30 & 31 & 32\\
\multicolumn{1}{|l|}{$x_{i}$} & 36 & 28 & 10 & 2 & 8 & 12 & 5 & 3 & 14 & 13 &
21 & 19 & 29 & 16 & 16 & 5\\
\multicolumn{1}{|l|}{$n_{i}$} & 40 & 33 & 25 & 5 & 23 & 28 & 20 & 28 & 20 &
23 & 41 & 32 & 39 & 28 & 21 & 24\\\hline
\multicolumn{1}{|l|}{$i$} & 33 & 34 & 35 & 36 & 37 & 38 & 39 & 40 & 41 & 42 &
43 & 44 & 45 & 46 & 47 & 48\\
\multicolumn{1}{|l|}{$x_{i}$} & 3 & 1 & 6 & 1 & 10 & 5 & 2 & 10 & 5 & 14 & 8 &
10 & 9 & 13 & 6 & 8\\
\multicolumn{1}{|l|}{$n_{i}$} & 30 & 15 & 23 & 5 & 35 & 30 & 14 & 56 & 51 &
62 & 45 & 55 & 42 & 27 & 36 & 31\\\hline
\multicolumn{1}{|l|}{$i$} & 49 & 50 & 51 & 52 & 53 & 54 & 55 & 56 & 57 & 58 &
59 & 60 & 61 & 62 & 63 & 64\\
\multicolumn{1}{|l|}{$x_{i}$} & 2 & 11 & 8 & 3 & 19 & 17 & 12 & 15 & 15 & 12 &
21 & 40 & 30 & 4 & 3 & 6\\
\multicolumn{1}{|l|}{$n_{i}$} & 9 & 26 & 38 & 29 & 55 & 37 & 45 & 47 & 44 &
45 & 33 & 65 & 85 & 13 & 9 & 16\\\hline
\end{tabular}%
\end{table}%

Since the proposed test-statistics are valid for single change-point
detection, now we are going to describe the algorithm based on the binary
segmentation procedure. In order to make a sequence of hypothesis testing, it
is convenient to use $\alpha=0.01$ if we want to get a not very large upper
bound for the global significance level according to the Bonferroni's
inequality. Suppose that the power-divergence based test-statistics with
$\lambda=2$, $\epsilon=0.05$, $^{0.05}T_{2}^{(K)}$, is our focus of interest.
The algorithm based on the binary segmentation procedure (Vostrikova (1981))
is described in Figures \ref{Segment1}-\ref{Segment2}. We consider
$N(\epsilon)=\{3,...,61\}$ as change point candidates in Step 1, \ i.e. we
have initially taken $\{1,...,K-1\}$\ but we have removed all candidates $k$
such that $k<K\epsilon$ or $k>K\epsilon$. Once the values of $T_{2}%
(\widehat{\theta}_{0,k}^{(K)},\widehat{\theta}_{1,k}^{(K)})$ are obtained for
each candidate belonging to $k\in N(\epsilon)$, we select its maximum
argument, $k=31$, which is accepted as change-point because the $p$-value is
less than $0.1$. The $p$-values are calculated by following (\ref{eqAprox}).
From now we have to investigate how to divide $[1,31]$ into segments (Step 2)
and $[32,64]$. We will continue until all candidates have $p$-values greater
than $0.1$. After $12$ steps it is concluded that the Lindisfarne Gospels
could have been written by seven scribes because the obtained segments are
$[1,10]$, $[11,18]$, $[19,23]$, $\{24\}$, $[25,31]$, $[32,52]$, $[53,64]$.
This conclusion differs a little bit from the conclusion obtained in
Horv\'{a}th and Serbinowska (1995), because the number of scribes they
proposed was one less and the locations of the change points are not exactly
the same.%

\begin{figure}[tbp] \centering
\begin{tabular}
[c]{ll}%
\raisebox{-0cm}{\fbox{\includegraphics[
height=5.5201cm,
width=8.2813cm
]%
{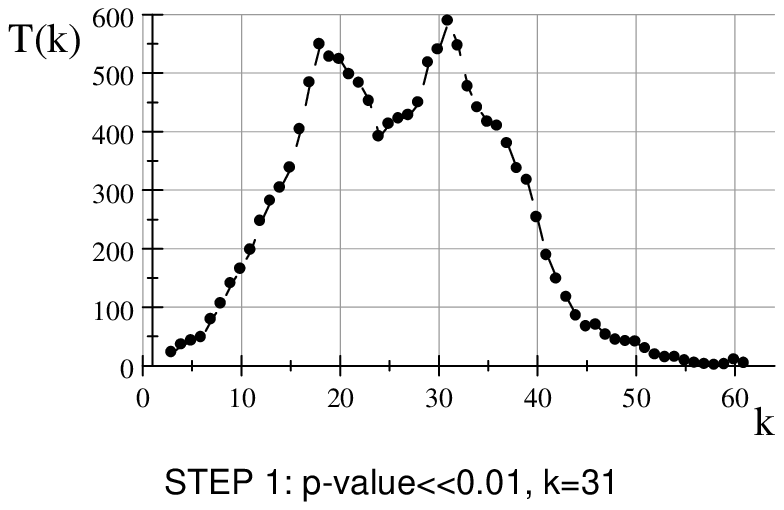}%
}}
\smallskip\negthinspace & \negthinspace%
\raisebox{-0cm}{\fbox{\includegraphics[
height=5.5201cm,
width=8.2813cm
]%
{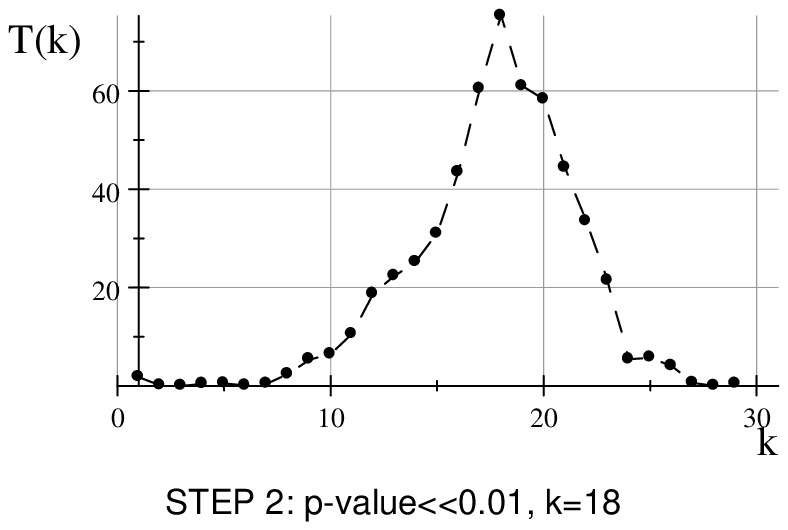}%
}}
\smallskip\\%
\raisebox{-0cm}{\fbox{\includegraphics[
height=5.5201cm,
width=8.2813cm
]%
{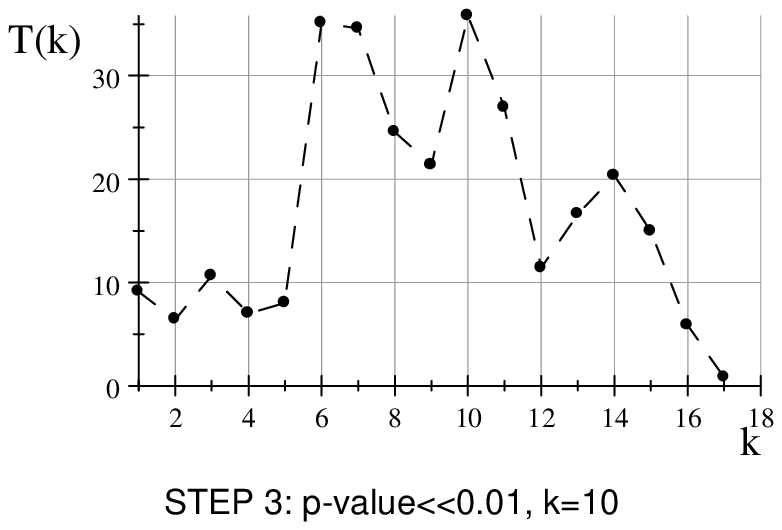}%
}}
\smallskip\negthinspace & \negthinspace%
\raisebox{-0cm}{\fbox{\includegraphics[
height=5.5201cm,
width=8.2813cm
]%
{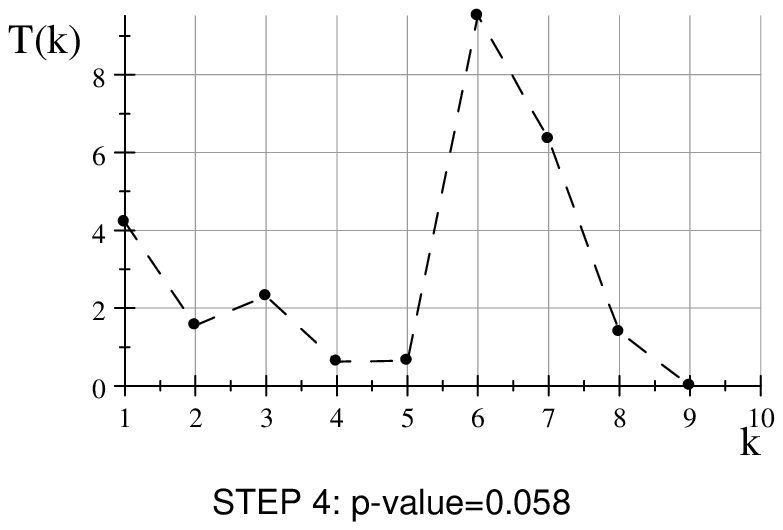}%
}}
\smallskip\\%
\raisebox{-0cm}{\fbox{\includegraphics[
height=5.5201cm,
width=8.2813cm
]%
{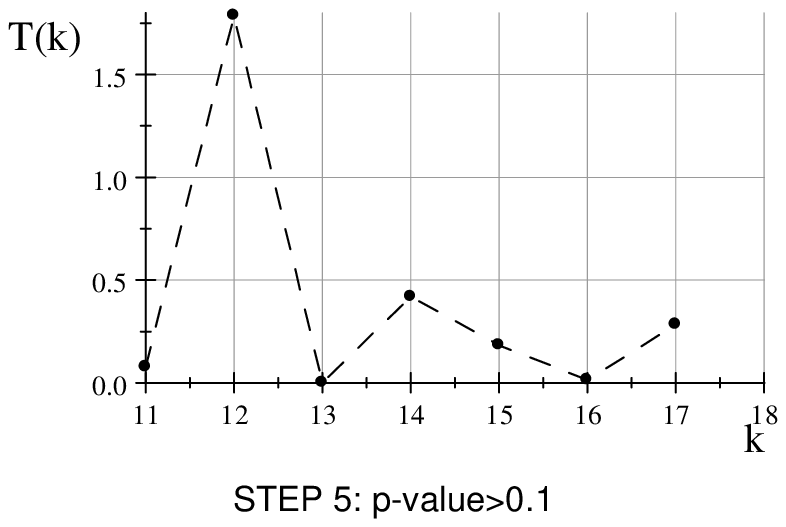}%
}}
\smallskip\negthinspace & \negthinspace%
\raisebox{-0cm}{\fbox{\includegraphics[
height=5.5201cm,
width=8.2813cm
]%
{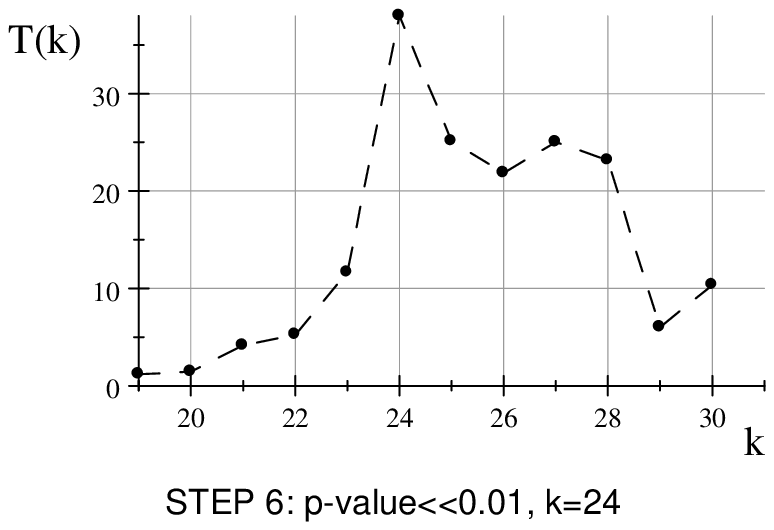}%
}}
\smallskip
\end{tabular}
\caption{Binary segmentation procedure for the Lindisfarne's problem (part
I)\label{Segment1}}
\end{figure}%
%

\begin{figure}[tbp] \centering
\begin{tabular}
[c]{ll}%
\raisebox{-0cm}{\fbox{\includegraphics[
height=5.5201cm,
width=8.2813cm
]%
{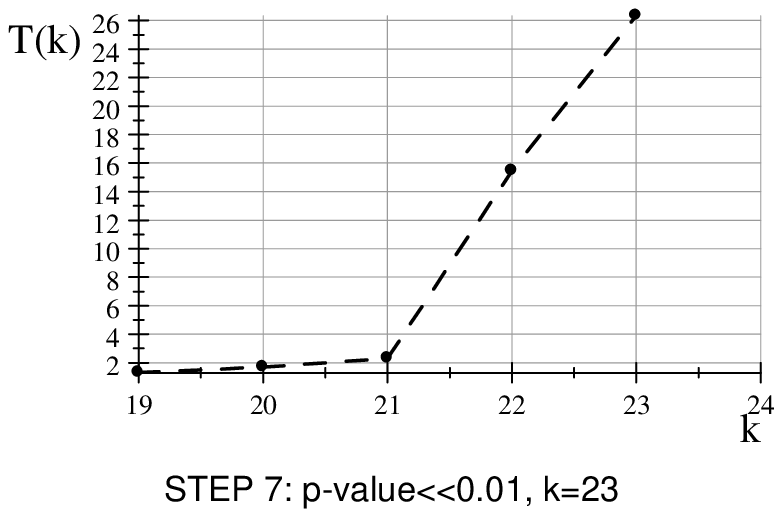}%
}}
\smallskip\negthinspace & \negthinspace%
\raisebox{-0cm}{\fbox{\includegraphics[
height=5.5201cm,
width=8.2813cm
]%
{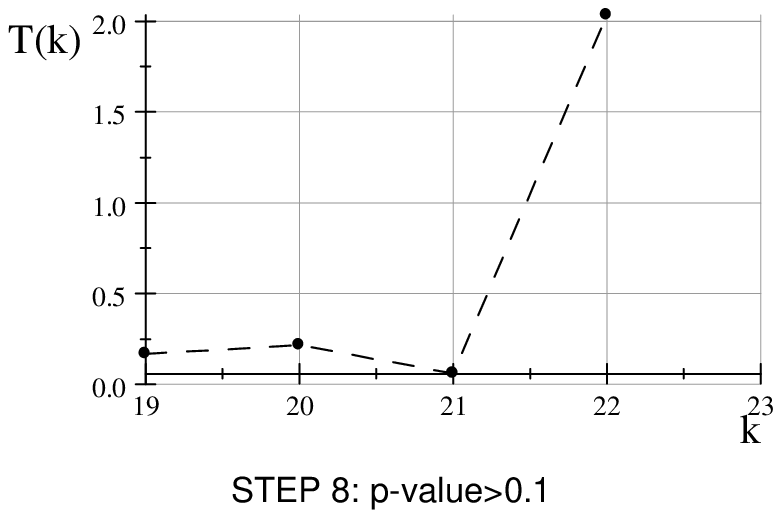}%
}}
\smallskip\\%
\raisebox{-0cm}{\fbox{\includegraphics[
height=5.5201cm,
width=8.2813cm
]%
{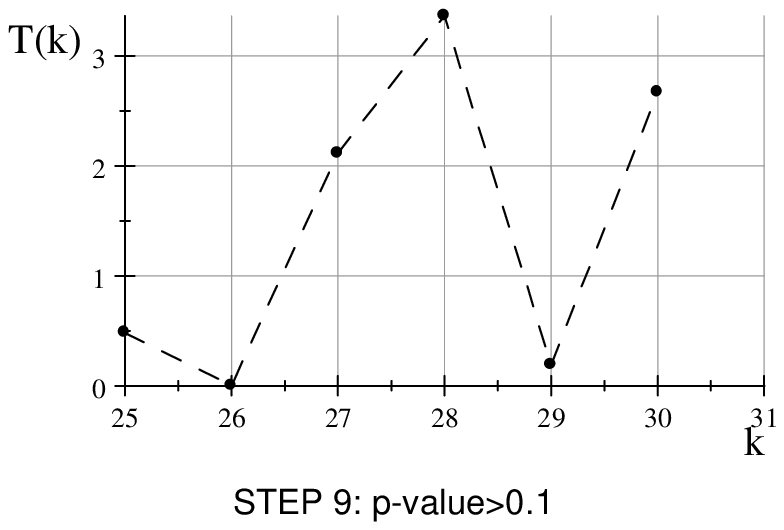}%
}}
\smallskip\negthinspace & \negthinspace%
\raisebox{-0cm}{\fbox{\includegraphics[
height=5.5201cm,
width=8.2813cm
]%
{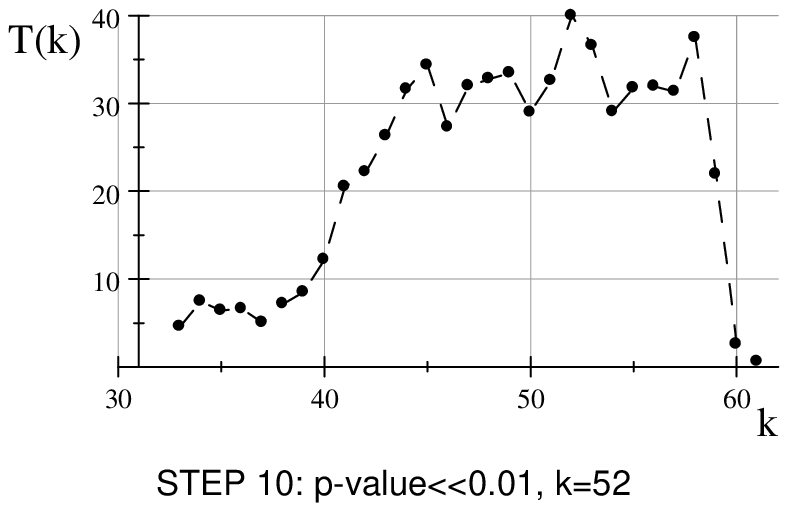}%
}}
\smallskip\\%
\raisebox{-0cm}{\fbox{\includegraphics[
height=5.5201cm,
width=8.2813cm
]%
{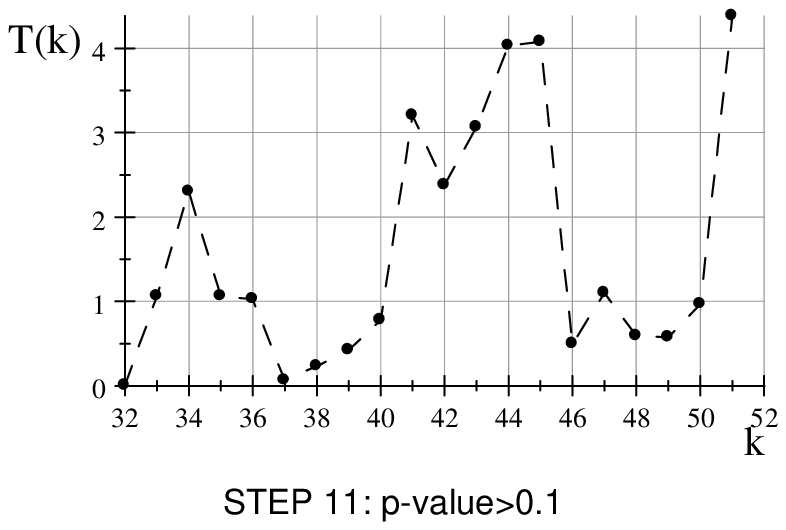}%
}}
\smallskip\negthinspace & \negthinspace%
\raisebox{-0cm}{\fbox{\includegraphics[
height=5.5201cm,
width=8.2813cm
]%
{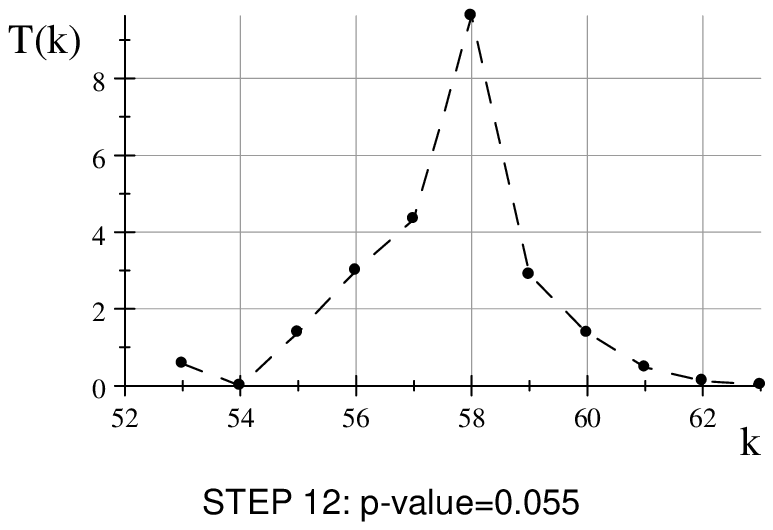}%
}}
\smallskip
\end{tabular}
\caption{Binary segmentation procedure for the Lindisfarne's problem (part
II)\label{Segment2}}
\end{figure}%

\end{document}